\documentclass[12pt]{amsart}
\usepackage[top=1in, bottom=1in, left=1in, right=1in]{geometry}
\usepackage{dsfont}
\usepackage{amsmath}
\usepackage{amsthm}
\usepackage{amssymb}
\usepackage{multirow}
\usepackage{mathtools}
\usepackage{times}
\usepackage[utf8]{inputenc}
\usepackage[english]{babel}

\usepackage{indentfirst}  
\usepackage{graphicx}
\usepackage{color}
\usepackage[colorlinks=true]{hyperref}
\hypersetup{colorlinks=true, citecolor=green, linkcolor=green, filecolor=magenta, urlcolor=cyan}
\usepackage[shortlabels]{enumitem}
\usepackage{titlesec}
\titleformat{\chapter}[display]{\bfseries\huge}{\filright\huge\chaptertitlename~\thechapter}{3ex}{\titlerule\vspace{1ex}\filright}[\vspace{1ex}\titlerule]

\def\Z{\mathbb{Z}}
\def\Q{\mathbb{Q}}

\newtheorem{theorem}{Theorem}
\newtheorem{lemma}[theorem]{Lemma}
\newtheorem{corollary}[theorem]{Corollary}
\newtheorem{proposition}[theorem]{Proposition}
\newtheorem{remark}[theorem]{Remark}

\newtheorem{definition}[theorem]{Definition}

\titleformat{\section}[block]{\scshape\centering}{\arabic{section}.}{1ex}{}{}
\titleformat{\subsection}[block]{\scshape\bfseries}{\thesubsection}{1ex}{}{}
\titleformat{\subsubsection}[block]{\itshape}{\thesubsubsection}{1ex}{}{}

\def\Z{\mathbb Z}

\def\Q{\mathbb Q}
\def\R{\mathbb R}

\def\cB{\mathcal B}
\def\cA{\mathcal A}

\def\cX{\mathcal X}

\def\cZ{\mathcal Z}

\def\Z{\mathbb Z}

\def\Q{\mathbb Q}
\def\R{\mathbb R}

\def\cB{\mathcal B}
\def\cA{\mathcal A}

\def\Re{\operatorname{Re}}

\def\mod{\operatorname{mod}}

\def\gcd{\operatorname{gcd}}

\def\log{\operatorname{log}}

\def\ord{\operatorname{ord}}

\usepackage[OT2,T1]{fontenc}
\DeclareSymbolFont{cyrletters}{OT2}{wncyr}{m}{n}
\DeclareMathSymbol{\Sha}{\mathalpha}{cyrletters}{"58}

\usepackage{cleveref}

\title{Ranks of Elliptic Curves Twisted by Quadratic Forms}
\author{Mohammad H. Hamdar}
\address{Department of Mathematics and Statistics, Concordia University, 1455 de Maisonneuve
West, Montreal, Quebec, Canada H3G 1M8, and D\'epartement de math\'ematiques et de statistique\\
	Universit\'e de Montr\'eal\\
	CP 6128 succ. Centre-Ville\\
	Montr\'eal, QC H3C 3J7\\
	Canada}
\curraddr{}
\email{mohammadhussein.hamdar@mail.concordia.ca}
\thanks{}

\author{Cihan Sabuncu}
\address{Max Planck Institut für Mathematik\\ 
Vivatsgasse 7 \\
53111 Bonn \\
Germany}
\curraddr{}
\email{sabuncu@mpim-bonn.mpg.de}
\thanks{}

\begin{document}

\maketitle
\begin{abstract}
    Let $E$ be an elliptic curve over $\Q$ and let $E^d$ be its twist by the quadratic character $\chi_d$. We prove there are infinitely many twists $d$ which are sums of two squares such that $E^d$ has rank $1$. This result is achieved using moments of derivatives of modular $L$-functions, and particularly captures the lower derivatives which were left out in the work of Munshi \cite{Munshi1}. Such a result, in particular, also gives us information on the elliptic fibration $(1+t^2)y^2=f(x)$, where $f(x)$ is a cubic polynomial. 
\end{abstract}
\section{Introduction and Main Results}

Let $E$ be an elliptic curve over $\Q$. The Birch and Swinnerton-Dyer (BSD) conjecture dictates that the analytic and algebraic ranks of $E$ are actually the same invariant. Recent years have seen substantial progress towards BSD, and in particular for small ranks. The first such breakthrough was due to Gross--Zagier \cite{grosszagier} and Kolyvagin \cite{Kolyvagin1} which together establish that if $E$ has analytic rank $0$ (respectively $1$), then $E$ has algebraic rank $0$ (respectively $1$) provided we know something about the  twist of $E$ by a quadratic character $\chi_d$, which we denote by $E^d$. In particular, what we want to know is that if there exists some $d$ such that $E^d$ has analytic rank $0$ or $1$, i.e. if the normalized $L$-function $L(s,E\otimes\chi_d)$ doesn't vanish or has a simple pole at $s=1/2$. Such statements were first proven by \cite{waldspurger, MurtyMurty,BumpFrieHoff}. The study of twists of $E$ has thus become one of the central areas in number theory since then, and other reasons include their relations to Fourier coefficients of automorphic forms, see for instance \cite{wald2, kohnenzagier}.

Studying $E^d$ when $d$ runs over twists in thinner families is a harder question. We mention for instance the result of Coates--Li--Tian--Zhai in \cite{Coates}, which proves the existence of infinitely many twists $d$ which are products of exactly $k$ prime factors such that $E^d$ has rank $0$ (or $1$), under certain conditions on $E$.

In this paper, we consider twists which are sums of two squares. Such a family of twists is non-linear and information on ranks of $E^d$ for these $d$'s would give us information about the elliptic surface $(1+t^2)y^2=f(x)$, where $f(x)$ is a cubic polynomial, via the Shioda--Tate formula for instance (cf. \cite{shioda}); in particular information on its Picard number and generic rank. This family was also considered by Munshi in \cite{Munshi1}, but results on ranks of $E^d$ don't follow from his methods. In the case $E$ has Complex Multiplication (CM) however, Munshi \cite{Munshi2} was able to get results about ranks. We will show such results in the general non-CM case, and thus extend the work of Munshi in \cite{Munshi1}. We first start by introducing some definitions and notations.

Let $f \in S_{\kappa}(N) $ be a normalized Hecke eigenform of even weight $\kappa$ and level $N$. Denote by $\lambda_f(n)$ the $n^{th}$ normalized Fourier coefficient of $f$, i.e. $$f(z)=e(z)+\sum_{n\geq 2}\lambda_f(n)n^{\frac{\kappa-1}{2}}e(nz).$$ Its $L$-function is defined by
\[L(s,f)=\sum_{n=1}^{\infty}\frac{\lambda_f(n)}{n^s}=\prod_{p\nmid N}\frac{1}{1-\lambda_f(p)p^{-s}+p^{-2s}}\prod_{p|N}\frac{1}{1-\lambda_f(p)p^{-s}}\]
for $\Re(s)>1$.
Define the completed $L$-function as
\[\Lambda(f,s)=N^{s/2}(2\pi)^{-s}\Gamma\left(s+\frac{\kappa-1}{2}\right)L(s,f).\]
This satisfies the functional equation
\[\Lambda(s,f)=i^{\kappa}\eta\,\Lambda(1-s,f),\]
where $\eta$ is the eigenvalue of the Fricke involution $W_N$ acting on $f$. We denote by $w(f)=i^{\kappa}\eta$ the root number of $f$.
The twisted $L$-function of $f$ by a fundamental discriminant $d$ with $(d,N)=1$ will satisfy
\[L(s,f\otimes\chi_d)=\sum_{n=1}^{\infty}\frac{\lambda_f(n)\chi_d(n)}{n^s}=\prod_{p\nmid dN}\frac{1}{1-\lambda_f(p)\chi_d(p)p^{-s}+p^{-2s}}\prod_{p|N}\frac{1}{1-\lambda_f(p)\chi_d(p)p^{-s}},\]
and if we set 
\[\Lambda(s,f\otimes\chi_d)=(|d|N)^{s/2}(2\pi)^{-s}\Gamma\left(s+\frac{\kappa-1}{2}\right)L(s,f\otimes \chi_d),\]
then we have the functional equation
\[\Lambda(s,f\otimes\chi_d)=i^{\kappa}\eta \chi_d(-N)\Lambda(1-s,f\otimes\chi_d).\]
Note that the level of the twisted modular form $f\otimes \chi_d$ becomes larger and in this case equals $Nd^2$.

Let $d$ be an odd, positive, squarefree (SF) integer. Then $\chi_{8d}(n)$ is a real primitive character with conductor $8d$ and satisfies $\chi_{8d}(-1)=1$. Let $$r(d)=\#\{(a,b)\in \Z^2\setminus (0,0): d=a^2+b^2\}$$ be the representation function for sums of two squares.  We prove the following result.
\begin{theorem}\label{MainTheorem}
Let $f\in S_{\kappa}(\Gamma_0(N))$ be a normalized eigenform, 
and $\psi$ be a smooth function with compact support in $(0,1]$. We have
    \[\sum_{\substack{d \ SF\\(d,2N)=1\\w(f\otimes\chi_{8d})=-1}}r(d)L'(1/2,f\otimes\chi_{8d})\psi\left(\frac{8d}{X}\right)=\mathcal{C}X\log X+O(X(\log X)^{1/2}(\log\log X)^3),\]
    with 
    \[\mathcal{C}=\frac{1}{2^7}\hat{\psi}(0)\prod_{p>2} \left( 1- \frac{\rho(p^2)}{p^4} \right)L(1,\text{sym}^2f)\Big[A(0,N,1)-w(f)A(0,N,N)\Big],
    \]
    where $\hat{\psi}$ is the Fourier transform of $\psi$, the function $\rho$ is defined in \eqref{def_of_rho}, $L(1,\text{sym}^2f)$ is defined in \eqref{sym^2}, and $A(0,N,N')$ is defined in \eqref{A(u)} for $N'\in\{1,N\}$. Note that $\mathcal{C}=0$ if and only if $w(f)=1$ and $N$ is a perfect square.
\end{theorem}
Such asymptotics for higher derivatives of the completed $L$-function $\Lambda^{(\ell)
}(1/2,f\otimes\chi_{8d})$ were obtained in \cite{Munshi1}, but Theorem \ref{MainTheorem} is one of the first results that deals with lower derivatives in such a family, and with the $L$-function itself rather than with the completed one.

A direct consequence of Theorem \ref{MainTheorem} is the following.

 \begin{corollary}\label{MainCorollary}
    Let $E$ be an elliptic curve over $\Q$ of odd conductor $N$. If $w(E)=1$, then assume also that $N$ is not a perfect square. There are infinitely many fundamental discriminants $d=a^2+b^2$ such that $E^d$ has algebraic rank $1$.
\end{corollary}

\begin{remark}
    Theorem \ref{MainTheorem} can be extended to any primitive positive-definite binary quadratic form $Q(x,y)=ax^2+bxy+cy^2$, 
    i.e. to study\footnote{In the former sum, the main term vanishes if the binary quadratic form $Q(u,v)$ is congruent to $1$ and $5 \pmod{8}$ equally often as $(u,v)$ ranges over residue classes modulo $8$.}
    $$ \sum_{\substack{d\equiv 1 \mod 4 \\ d \ SF\\(d,N)=1 \\w(f\otimes\chi_{d})=-1}}r_Q(d)L'(1/2,f\otimes\chi_{d})\psi\left(\frac{d}{X}\right), $$
    and
    $$ \sum_{\substack{ d \ SF\\(d,2N)=1 \\w(f\otimes\chi_{8d})=-1}}r_Q(d)L'(1/2,f\otimes\chi_{8d})\psi\left(\frac{8d}{X}\right), $$
    where
    $$ r_Q(d) := \{ (a,b)\in \Z^2 \setminus (0,0) : d=Q(a,b) \} .$$
    Specifically, for Proposition \ref{Voronoi}, one has a general Voronoi summation formula for $r_Q(d)$ (see Remark \ref{remark_general_quad_forms}). The additional modifications will be in the other results of Section \ref{sec-voronoi}, where one has to be careful incorporating $Q(x,y)$ (and, for the former sum, the case of even entries coming from the $L$-function\footnote{Writing $n=2^{\alpha}n'$ with $n'$ odd, isolates the powers of $2$, which imposes additional restrictions on the Gauss sums. The resulting sum over $\alpha$ will converge, and the argument proceeds similarly as in the odd case treated in the paper.}) into the study of the general two-dimensional Gauss sums defined in Remark \ref{remark_general_quad_forms}.
    
\end{remark}

The first moment of $L'$ but over the full family (i.e. without the $r(d)$ weight) was obtained in many papers, cf. \cite{MurtyMurty,iwaniec}. The first moment of $L$ over the full family instead of $L'$ was established in \cite{iwaniec, radzsound, shen22}. Asymptotics for second moments over the full family of twists can be found under GRH in \cite{soundyoung, Petrow}, and unconditionally in the work of Li \cite{XLi}. Following Li's ideas, unconditional mixed second moments were further done in \cite{Kumar,zhou, huang2025averagingquadraticallytwistedlvalues, chen26}. For analogous moment results over function fields, see \cite{florea}. 

\subsection{Strategy of the Proof}
The proof follows the general strategy as that of Soundararajan and Young \cite{soundyoung}. We begin with the approximate functional equation (Lemma \ref{AFE}),
$$
L'(1/2,f\otimes \chi_d)
\approx
\sum_{n\ge 1}\frac{\lambda_f(n)\chi_d(n)}{n^{1/2}}
W\left(\frac{n}{d}\right).
$$
We then split the sum over $n$ into a short range and a long range. For the short range, we introduce the weight $W(n/U)$, where
$$
U=\frac{X}{(\log X)^A},
$$
with $A>0$ sufficiently large. The long range is then captured by the complementary weight
$$
W\left(\frac{n}{d}\right)-W\left(\frac{n}{U}\right).
$$

For the short range, after temporarily ignoring the squarefree and coprimality conditions, we are led to study sums of the form
$$
\sum_{\substack{d\ge1\\ w(f\otimes\chi_{8d})=-1}}
r(d)
\sum_{n\ge1}
\frac{\lambda_f(n)\chi_{8d}(n)}{n^{1/2}}
h(d,n),
$$
where $h(x,y)$ is a smooth function, compactly supported in the $x$-variable. After interchanging the order of summation, we apply the Voronoi summation formula of $r(d)$ (Proposition \ref{Voronoi}). Since we employ Voronoi summation rather than Poisson summation, the resulting dual sum has expected length squared; consequently, the analysis resembles that of a second moment rather than a first moment.

This reduces the problem to estimating sums of the form
$$
\sum_{(k_1,k_2)\in\mathbb Z^2}
(-1)^{k_1+k_2}
G_{(k_1,k_2)}(n)
\check{h}\left(
\frac{(k_1^2+k_2^2)X}{2n^2},
n
\right),
$$
where
$$
G_{(k_1,k_2)}(n)
=
\sum_{b_1,b_2\bmod n}
\left(\frac{b_1^2+b_2^2}{n}\right)
e\left(
\frac{k_1b_1}{n}
+
\frac{k_2b_2}{n}
\right)
$$
is a two-dimensional Gauss sum, and $\check{h}(\xi,y)$ denotes the Bessel transform of $h(x,y)$ with respect to the $x$-variable (see \eqref{bessel_transform_def}).

The zero frequency $(k_1,k_2)=(0,0)$ produces the main term, following an argument analogous to that of Soundararajan and Young. For the nonzero frequencies, we invoke Mellin inversion to $\check{h}$, expressing it as a double contour integral. After a suitable factorization of $k_1$ and $k_2$, the resulting expression can be recast in terms of sums of central twists $L(1/2+it,f\otimes\chi_d)$. We then employ Li's estimate \cite{XLi} to obtain the required bounds.

Finally, the contribution from the long range is controlled using the work of Zhou \cite[Proposition~4.3]{zhou} via an application of Cauchy--Schwarz.

\begin{remark}
    Concurrently with the completion of this paper, a paper by Huang and Wei \cite{huang2026momentderivativeslfunctionsnonlinear} just appeared where they establish a similar result for the convolutions $1 * \chi_D$, which in particular recovers an analogue of Theorem~\ref{MainTheorem}. Their work does not consider the family subject to the root number condition on the twist; however, with their restriction on the discriminants, they can get an analogue of Corollary \ref{MainCorollary} subject to a global root number condition on the elliptic curve. Our approach was developed independently and differs substantially in its treatment of the frequency side, which constitutes the principal point of departure; the treatment of the error terms in the remainder of the argument is similar in spirit. More precisely, Huang and Wei employ Mellin inversion to study averages of Dirichlet $L$-functions twisted by the Fourier coefficients $\lambda_f(n)$. By contrast, our argument follows the classical Poisson summation philosophy, specifically through a Voronoi summation formula, which leads to the analysis of certain two-dimensional Gauss sums. The Voronoi summation formula established in Proposition~\ref{Voronoi} may be of independent interest, together with the understanding of the properties of these relevant Gauss sums which is the content of Section \ref{sec-voronoi}.
\end{remark}

\textbf{Acknowledgements.} The authors are grateful to Valentin Blomer, Chantal David, Andrew Granville, Xiannan Li, and Maksym Radziwi\l\l{} for their valuable advice and insightful comments. They also thank Stephanie Chan and Oleksiy Klurman for their helpful suggestions. Part of this work was carried out during the \emph{Probability in Number Theory} workshop at the Centre de Recherches Mathématiques. The authors gratefully acknowledge the excellent working conditions and hospitality provided by the Centre de Recherches Mathématiques.

\subsection{Notation} 
We use the usual Vinogradov asymptotic notation. Throughout the paper, $r(d)$ is the representation function for the sum of two squares, and $J_0(x)$ is the usual Bessel function of the first kind. We use $\left( \frac{\cdot}{\cdot} \right)$ for the Kronecker symbol, and $e(x) = e^{2i\pi x}$. We use $\tilde{h}$ for the Mellin transform, $\hat{f}$ for the Fourier transform, and $\check{g}$ for the Bessel transform as defined in \eqref{bessel_transform_def}. Further, throughout the paper, we use $\tau(n)$ for the divisor function. We will usually write $\tau(n)^{O(1)}$ to denote an unspecified fixed power of the divisor function. In certain parts of the paper, we use $\tau_p$ to denote the usual Gauss sum $\mod p$, namely 
$$ \tau_p =\sum_{x \mod p}\left( \frac{x}{p} \right) e\left( \frac{x}{p} \right) .$$

\section{Preliminaries}
In this section, we will collect certain Lemmas which we will find useful throughout.
\begin{lemma}\label{AFE}
    Let $f\in S_{\kappa}(\Gamma_0(N))$ be a normalized eigenform with root number $w(f)=i^{\kappa}\eta$. We have that 
    \begin{align*}
    L'(1/2,f\otimes \chi_d)=\sum_{n\geq 1}\frac{\lambda_f(n)\chi_d(n)}{n^{1/2}}W\left(\frac{n}{|d|}\right)+i^{\kappa}\eta\chi_d(-N)\sum_{n\geq 1}\frac{\lambda_f(n)\chi_d(n)}{n^{1/2}}W\left(\frac{n}{|d|}\right),
    \end{align*}
    where $W(y)$ is a smooth function defined by
    \[W(y)=\frac{1}{2\pi i}\int_{(3)}\frac{\Gamma(u+\kappa/2)}{\Gamma(\kappa/2)}\left(\frac{2\pi y}{\sqrt{N}}\right)^{-u}\,\frac{du}{u^2}.\]
\end{lemma}
\begin{proof}
    This follows from the classical proof for the $L$-function $L(s,f)$ in \cite[Theorem 5.3]{Iwk}. We refer the reader to \cite[Lemma 1]{Petrow} for a complete proof for the derivative $L'(s,f)$.
\end{proof}
 The function $W$ and its derivatives $W^{(j)}$ satisfy the following bounds.
 
\begin{lemma}\label{W_bound}
For any $y\in \R$ and $j\in\Z_{\geq 1}$, we have that
  \[W^{(j)}(y)\ll_{j,B} \frac{1}{y^j}\left(\log\frac{\kappa N}{y}\right)\left(1+\frac{y}{\kappa \sqrt{N}}\right)^{-B}\]
  for any $B>0$.
\end{lemma}

\begin{proof}
    The bounds follow in the same way as in the proof of \cite[Proposition 5.4]{Iwk}, taking into consideration the extra $y^{-1}$ factor in the defining integral.
\end{proof}
One of the important consequences of Li's work \cite{XLi} is the following unconditional bound for the second moment of the twisted $L$-function.
\begin{lemma}\label{XLi_bound}
Let $f\in S_{\kappa}(\Gamma_0(N))$. Then we have the bound
    $$\sum_{0<d\leq X} \left|L\left(\frac{1}{2}+it , f \otimes \chi_d \right)\right|^2 \ll_f (1+|t|)^{2}X \log X .$$
\end{lemma}
\begin{proof}
    This follows from \cite[Proof of Theorem 1.1]{XLi}.
\end{proof}

 \section{Voronoi Summation for Sums of Two Squares and Two-Dimensional Gauss Sums}\label{sec-voronoi}
We first start by proving a Voronoi summation formula that incorporates the representation function $r(d)$.
\begin{proposition}[Voronoi Summation Formula]\label{Voronoi}
Let $F$ be a smooth function with compact support on the positive real numbers, and $n$  an odd integer. Let also $M$ be a positive integer coprime to $2n$. Then
$$ \sum_{\substack{(d,2)=1 \\ d\equiv 0 \mod M}} r(d) \left( \frac{d}{n} \right) F\left( \frac{d}{Z} \right) = \frac{Z}{4M^2n^2} \left( \frac{2}{n} \right) \sum_{(k_1,k_2)\in \Z^2} (-1)^{k_1+k_2} G_{(k_1,k_2)}(M, n) \check{F}\left( \frac{(k_1^2+k_2^2)Z}{2M^2n^2} \right) 
$$
where
    $$ G_{(k_1,k_2)}(M,n) := \sum_{\substack{b_1,b_2 \mod Mn \\ b_1^2+b_2^2\equiv 0 \mod M}} \left( \frac{b_1^2+b_2^2}{n}\right) e\left( \frac{k_1b_1}{Mn} + \frac{k_2 b_2}{Mn} \right), $$
and
\begin{equation}\label{bessel_transform_def}
\check{F} (u) = 2\pi \int_0^{\infty} F(t) J_0(2\pi \sqrt{u t}) dt ,
\end{equation}
is the Bessel transform of $F$.
\end{proposition}
\begin{proof}
    We start by getting rid of the $(d,2)=1$ condition;
    \begin{equation}\label{splitting_odd}
    \sum_{\substack{(d,2)=1 \\ d\equiv 0 \mod M}} r(d) \left( \frac{d}{n} \right) F\left( \frac{d}{Z} \right) = \sum_{ d\equiv 0 \mod M} r(d) \left( \frac{d}{n} \right) F\left( \frac{d}{Z} \right) - \left( \frac{2}{n} \right) \sum_{d\equiv 0 \mod M} r(d) \left( \frac{d}{n} \right) F\left( \frac{2d}{Z} \right),
    \end{equation}
    since $r(2d)=r(d)$. Next, we focus on the sum
    $$ \sum_{d\equiv 0 \mod M} r(d) \left( \frac{d}{n} \right) F\left( \frac{ad}{Z} \right) ,$$
    for $a=1$ or $2$. We take out the real character by imposing congruence conditions on  $a_1$ and $a_2$ where $a_1^2+a_2^2=d$. This gives us, since $(M,2n)=1$, that the latter sum equals
    $$ \sum_{\substack{b_1,b_2 \mod Mn \\ b_1^2+b_2^2\equiv 0\mod M}} \left( \frac{b_1^2+b_2^2}{n}\right) \sum_{\substack{a_1\equiv b_1 \mod Mn \\ a_2\equiv b_2 \mod Mn}} F\left(\frac{(a_1^2+a_2^2) a}{Z} \right). $$
    The inner sum over $a_1,a_2$ by the Voronoi summation formula gives (by writing $a=b_1 + Mn \ell_1$ and $b=b_2 + Mn \ell_2$)
    $$ \sum_{\substack{a_1\equiv b_1 \mod Mn \\ a_2\equiv b_2 \mod Mn}} F\left(\frac{(a_1^2+a_2^2) a}{Z} \right) = \frac{Z}{2aM^2n^2} \sum_{(k_1,k_2)\in \Z^2} \check{F}\left(  \frac{(k_1^2+k_2^2)Z}{a M^2n^2}\right) e\left( \frac{k_1b_1}{Mn} + \frac{k_2 b_2}{Mn} \right) ,$$
    where
    $ \check{F}(u)$ is the Bessel transform of $F$ given in \eqref{bessel_transform_def}. 
    Now summing over $b_1,b_2 \mod Mn$, we get that
    $$ \sum_{d\equiv 0 \mod M} r(d) \left( \frac{d}{n} \right) F\left( \frac{ad}{Z} \right)  = \frac{Z}{2aM^2n^2} \sum_{(k_1,k_2)\in \Z^2} G_{(k_1,k_2)}(M,n) \check{F}\left(  \frac{(k_1^2+k_2^2)Z}{a M^2n^2}\right),  $$
    where $G_{(k_1,k_2)}(M,n)$ is as in the statement of the result. Putting this in \eqref{splitting_odd}, we write $k_1=\frac{k_1'+k_2'}{2}$, and $k_2=\frac{k_1'-k_2'}{2}$ for $k_1'\equiv k_2'\mod 2$, and use the following relation that we prove below $$G_{\left(\frac{k_1'+k_2'}{2},\frac{k_1'-k_2'}{2}\right)}(M, n) = \left(\frac{2}{n}\right) G_{(k_1',k_2')}(M, n),$$ to get the result.

    To prove $G_{\left(\frac{k_1+k_2}{2},\frac{k_1-k_2}{2}\right)}(M,n) = \left(\frac{2}{n}\right) G_{(k_1,k_2)}(M, n)$ for $k_1\equiv k_2 \mod 2$, we note that
$$ \left( \frac{2}{n}\right) G_{(k_1,k_2)}(M, n) = \sum_{\substack{b_1,b_2 \mod Mn \\ b_1^2+b_2^2\equiv 0 \mod M}} \left( \frac{2(b_1^2+b_2^2)}{n}\right)  e\left( \frac{k_1b_1}{Mn} + \frac{k_2 b_2}{Mn} \right) .$$
We can then do a change of variables $b_1 \to b_1+b_2$ and $b_2 \to b_1 - b_2$ which yields 
$$ \left( \frac{2}{n}\right) G_{(k_1,k_2)}(M, n) = \sum_{\substack{b_1,b_2 \mod Mn \\ b_1^2+b_2^2\equiv 0 \mod M}} \left( \frac{b_1^2+b_2^2}{n}\right)  e\left( \frac{\bar{2}(k_1+k_2)b_1}{Mn} + \frac{\bar{2}(k_1-k_2) b_2}{Mn} \right), $$
where $\bar{2}$ denotes the inverse of $2 \mod Mn$. 
Next, by reciprocity we have
$$ \frac{\bar{2}}{Mn} \equiv \frac{1}{2Mn} - \frac{1}{2} \mod 1. $$
Plugging this in, we get
\begin{align*}
\left( \frac{2}{n}\right) G_{(k_1,k_2)}(M, n) &= \sum_{\substack{b_1,b_2 \mod Mn \\ b_1^2+b_2^2\equiv 0 \mod M}} \left( \frac{b_1^2+b_2^2}{n}\right)  e\left( \frac{(k_1+k_2)b_1}{2Mn} + \frac{(k_1-k_2) b_2}{2Mn} \right) \\
&= G_{\left(\frac{k_1+k_2}{2},\frac{k_1-k_2}{2}\right)}(M, n),
\end{align*}
since $k_1\equiv k_2 \mod 2$, and so the term coming from the $-\frac{1}{2}$ factor of the reciprocity law disappears.
\end{proof}

\begin{remark}\label{remark_general_quad_forms}
Following a similar argument for any primitive positive-definite binary quadratic form $Q(x,y)=ax^2+bxy+cy^2$ with discriminant $\Delta$, one obtains a general Voronoi summation formula of the form
    $$  \sum_{\substack{d\equiv h \mod M}} r_Q(d) \left( \frac{d}{n} \right) F\left( \frac{d}{Z} \right) = \frac{Z}{\sqrt{|\Delta|}M^2n^2} \sum_{(k_1,k_2)\in \Z^2} G_{(k_1,k_2)}(M, n, h; Q) \check{F}\left( \frac{4Z\cdot Q^*(k_1,k_2)}{|\Delta|M^2n^2} \right) , $$
    where $Q^*(x,y) = cx^2-bxy+ay^2$ is the dual quadratic form, $\check{F}(x)$ is the Bessel transform defined in \eqref{bessel_transform_def}, and 
    $$ G_{(k_1,k_2)}(M, n, h ; Q) := \sum_{\substack{b_1, b_2 \mod M n\\ Q(b_1,b_2) \equiv h \mod M}} \left( \frac{Q(b_1,b_2)}{n} \right) e\left( \frac{k_1 b_1 + k_2 b_2}{Mn} \right),$$
    are the two-dimensional Gauss sums attached to $Q(x,y)$.
    \end{remark}

The next lemma shows a convenient splitting of the two-dimensional Gauss sums we encountered in Proposition \ref{Voronoi}. In particular, it turns out that one may separate the local conditions appropriately.
\begin{lemma}\label{splitting_G}
Let $k_1,k_2\in \Z$, and $n$ coprime to $M$. Then we have that
$$ G_{(k_1,k_2)}(M,n)= G_{(k_1,k_2)}(n) \cdot \rho_{(k_1,k_2)}(M) ,$$
where 
$$ G_{(k_1,k_2)}(n) := \sum_{b_1,b_2\mod n} \left(\frac{b_1^2+b_2^2}{n}\right) e\left( \frac{k_1b_1}{n} + \frac{k_2 b_2}{n} \right) , $$
and
$$\rho_{(k_1,k_2)}(M) := \sum_{\substack{b_1,b_2 \mod M \\ b_1^2+b_2^2\equiv 0 \mod M}} e \left( \frac{k_1b_1}{M} + \frac{k_2 b_2}{M} \right). $$
\end{lemma}
\begin{proof}
    We start by noting that if $(b_1,b_2)\in \Z/Mn\Z$ is such that $b_1^2+b_2^2 \equiv 0 \mod M$, then by the Chinese remainder theorem, one can write $b_1=b_1'M+b_1''n$ for $b_1' \mod n$ and $b_1'' \mod M$ since $\gcd(n,M)=1$, and similarly for $b_2$. Using this, we get
    \begin{align*}
        &G_{(k_1,k_2)}(M,n) \\
        &= \sum_{b_1',b_2' \mod n} \sum_{\substack{b_1'',b_2'' \mod M \\ b_1''^2+b_2''^2 \equiv 0 \mod M}} \left( \frac{(b_1'M+ b_1''n)^2 + (b_2'M+b_2''n)^2}{n} \right) e\left( \frac{k_1(b_1'M+b_1''n)}{Mn} + \frac{k_2(b_2'M+b_2''n)}{Mn} \right) \\
        &=\sum_{b_1',b_2' \mod n} \left( \frac{b_1'^2+b_2'^2}{n} \right) e\left( \frac{k_1 b_1'}{n} + \frac{k_2 b_2'}{n}\right) \cdot \sum_{\substack{b_1'',b_2'' \mod M \\ b_1''^2+b_2''^2 \equiv 0 \mod M}} e\left( \frac{k_1 b_1''}{M} + \frac{k_2 b_2''}{M}\right),
    \end{align*}
    which gives us the result.
\end{proof}

Denote by 
\begin{equation}\label{def_of_rho}
\rho(n):=\rho_{(0,0)}(n)=\#\{ (u_1,u_2)\mod n : u_1^2+u_2^2\equiv 0 \mod n \}.    
\end{equation}

We now prove multiplicative properties of our Gauss sums and how they behave locally.
\begin{proposition}\label{G_k(n) conditions}
If $m$ and $n$ are relatively prime odd integers, then $$G_{(k_1,k_2)}(mn)=G_{(k_1,k_2)}(m)G_{(k_1,k_2)}(n).$$ Moreover, if $p^{\alpha}$ is the highest power of $p$  dividing $\gcd(k_1,k_2)$ (setting $\alpha=\infty$ if $(k_1,k_2)=(0,0)$), then we have for $\beta\geq 1$,
$$ G_{(k_1,k_2)}(p^{\beta}) = \begin{cases}
    0 & \text{if } \beta\leq \alpha \text{ is odd} \\
    p^{2(\beta-1)} (p^2 - \rho(p)) & \text{if } \beta\leq \alpha \text{ is even} \\
    p^{2\alpha+1} \left( \frac{(k_1^2+k_2^2)p^{-2\alpha}}{p}\right) & \text{if } \beta = \alpha + 1 \text{ is odd} \\ 
    -p^{2\alpha}\left( \frac{-1}{p} \right)(p\cdot \delta_p((k_1^2+k_2^2)p^{-2\alpha}) - 1)  & \text{if } \beta=\alpha + 1 \text{ is even} \\
    0 & \text{if } \beta\geq \alpha+2,
\end{cases} $$
where $\delta_p(n)$ is $1$ if $n\equiv 0 \mod p$ and $0$ otherwise, and $\rho(n)$ is as in \eqref{def_of_rho}.
\end{proposition}
\begin{proof}
    We first prove the multiplicativity. Since $\gcd(m,n)=1$, write $b_1 = m b_1'+ nb_1'' $ for $b_1' \mod n$ and $b_1'' \mod m$ by the Chinese remainder theorem. And similarly for $b_2=m b_2' + n b_2''$. We then have
    $$G_{(k_1,k_2)}(mn) = \sum_{\substack{b_1', b_2' \mod n \\ b_1'',b_2'' \mod m}}\left(\frac{(b_1' m+ b_1''n)^2+(b_2' m + b_2''n)^2}{mn} \right) e\left( \frac{k_1(b_1' m + b_1'' n) + k_2 (b_2' m + b_2'' n)}{mn} \right). $$
    Using the multiplicativity of the Kronecker symbol and the linearity in the exponential phase, we get that this is $G_{(k_1,k_2)}(m) G_{(k_1,k_2)}(n)$ as desired.

    We now prove the second part of the lemma. If $\beta=\alpha + 1$, then
    \begin{align*}
    \sum_{b_1,b_2 \mod p^{\beta}} \left( \frac{b_1^2+b_2^2}{p^{\beta}} \right) &e\left( \frac{k_1 b_1}{p^{\beta}} + \frac{k_2 b_2}{p^{\beta}}\right) \\
    &= \sum_{\ell_1 , \ell_2 \mod p} \left( \frac{\ell_1^2+\ell_2^2}{p^{\beta}} \right) \sum_{a_1 , a_2 \mod p^{\beta-1}} e\left( \frac{k_1 (a_1p + \ell_1) + k_2 (a_2p + \ell_2)}{p^{\beta}} \right) \\
    &= p^{2(\beta-1)} \sum_{\ell_1 , \ell_2 \mod p} \left( \frac{\ell_1^2+\ell_2^2}{p^{\beta}} \right) e\left( \frac{k_1 \ell_1 + k_2 \ell_2}{p^{\beta}} \right).
    \end{align*}
    If $\beta$ is even, then the sum above is
    $$ \sum_{\substack{\ell_1 , \ell_2 \mod p \\ \ell_1^2+\ell_2^2 \not\equiv 0 \mod p}} e\left( \frac{k_1 \ell_1 + k_2 \ell_2}{p^{\beta}} \right) = - \sum_{\substack{\ell_1 , \ell_2 \mod p \\ \ell_1^2+\ell_2^2 \equiv 0 \mod p}} e\left( \frac{k_1 \ell_1 + k_2 \ell_2}{p^{\beta}} \right) ,$$
    since the full sum is 0. We write $k_1'=k_1/p^{\alpha}$ and $k_2'= k_2/p^{\alpha}$. By orthogonality, we get
    $$ \sum_{\substack{\ell_1 , \ell_2 \mod p \\ \ell_1^2+\ell_2^2 \equiv 0 \mod p}} e\left( \frac{k_1 \ell_1 + k_2 \ell_2}{p^{\beta}} \right) = \frac{1}{p} \sum_{a\mod p}\sum_{\substack{\ell_1 , \ell_2 \mod p}} e\left( \frac{a(\ell_1^2+\ell_2^2) + k_1' \ell_1 + k_2' \ell_2}{p} \right) .$$
    This equals
    $$ \frac{1}{p}\sum_{a\mod p} \left( \sum_{\ell_1 \mod p} e\left( \frac{a\ell_1^2+ k_1'\ell_1}{p} \right) \right) \left(\sum_{\ell_2 \mod p} e\left( \frac{a\ell_2^2+ k_2'\ell_2}{p} \right) \right) . $$
    Note that the term $a=0$ gives $0$. By the theory of Gauss sums, for $a\not=0$, these sums over $\ell_1, \ell_2$ evaluate to
    $$ \left( \frac{a}{p} \right) \tau_p e\left( \frac{-\overline{4a}k_i'^2}{p} \right), $$
    where $\tau_p$ is the usual Gauss sum and so $\tau_p^2 = \left( \frac{-1}{p} \right)p$.
    Hence, we get the fourth case of the lemma, since we are left with
    $$ - p^{2(\beta-1)} \left( \frac{-1}{p} \right) \sum_{\substack{a \mod p \\ a\not= 0}}e\left( \frac{-\overline{4a}(k_1'^2+k_2'^2)}{p} \right) , $$
    and this sum is the Ramanujan sum.

    If $\beta$ is odd, we have
    $$ p^{2(\beta-1)} \sum_{\ell_1 , \ell_2 \mod p} \left( \frac{\ell_1^2+\ell_2^2}{p} \right) e\left( \frac{k_1 \ell_1 + k_2 \ell_2}{p^{\beta}} \right) . $$
    Similar to the even case, we can write the sum as
    $$ \sum_{\ell_1 , \ell_2 \mod p} \left( \frac{\ell_1^2+\ell_2^2}{p} \right) e\left( \frac{k_1' \ell_1 + k_2' \ell_2}{p} \right) .$$
    We rewrite the Legendre symbol as
    $$ \left( \frac{\ell_1^2+\ell_2^2}{p} \right) = \frac{1}{\tau_p} \sum_{\substack{a \mod p \\ a\not=0}} \left( \frac{a}{p}\right) e\left( \frac{a(\ell_1^2+\ell_2^2)}{p} \right) . $$
    Putting this in, similar to the even case, we get sums of the form
    $$ \sum_{\ell_i \mod p} e\left( \frac{a\ell_i^2 + k_i' \ell_i}{p} \right) = e\left( - \frac{\overline{4a}k_1^2}{p} \right) \sum_{\ell_1} e\left( \frac{a(\ell_i + \overline{2a} k_i')^2}{p} \right). $$
    Now changing the variable of summation $\ell_i + k_i' \overline{a} \to \ell_i$, the inner sum evaluates to $\left( \frac{a}{p} \right)\tau_p$ by the classical theory of Gauss sums. Putting everything back and using $\tau_p^2 = \left( \frac{-1}{p} \right)p $ again, yields
    $$ \sum_{\ell_1 , \ell_2 \mod p} \left( \frac{\ell_1^2+\ell_2^2}{p} \right) e\left( \frac{k_1' \ell_1 + k_2' \ell_2}{p} \right) = \frac{\left(\frac{-1}{p} \right) p}{\tau_p} \sum_{\substack{a \mod p \\ a\not=0}} \left( \frac{a}{p} \right) e\left( \frac{-\overline{4a}(k_1'^2+k_2'^2)}{p} \right) .$$
    Lastly, we do the change of variables $-\overline{4a} \to u$ which gives $\left( \frac{a}{p} \right) = \left( \frac{-\overline{4u}}{p} \right) = \left( \frac{-1}{p} \right)\left( \frac{u}{p} \right)$. Hence, we obtain
    $$ \frac{p}{\tau_p} \sum_{\substack{u \mod p \\ u\not=0}} \left(\frac{u}{p} \right) e\left( \frac{u(k_1'^2+k_2'^2)}{p} \right) .$$
    Evaluating this Gauss sum gives the third case of the lemma.

    The other cases are done in a similar manner.
\end{proof}

\begin{definition}
 We say a function $f(x,y)$ over $\Z^2$ is jointly multiplicative if $f(m_1n_1,m_2n_2)=f(m_1,m_2)f(n_1,n_2)$ whenever $\gcd(m_1m_2,n_1n_2)=1$.
 \end{definition}
 The following observation will be found useful.

\begin{lemma}\label{jointlymultiplicative}
    Let $\ell,n$ be odd integers and $(k_1,k_2)\in\Z^2\setminus (0,0)$. Then $G_{(\ell k_1,\ell k_2)}(n)$ is jointly multiplicative in $\ell$ and $n$. 
\end{lemma}
\begin{proof}
    Define $f(n, \ell) = G_{(\ell k_1, \ell k_2)}(n)$. We want to show that this is jointly multiplicative. Given $\gcd(n_1\ell_1 , n_2 \ell_2)=1$, note that we have
    $$ f(n_1n_2, \ell_1 \ell_2) = \sum_{b_1, b_2 \mod n_1n_2} \left(\frac{b_1^2+b_2^2}{n_1n_2} \right) e\left( \frac{\ell_1 \ell_2 k_1 b_1 + \ell_1 \ell_2 k_2 b_2}{n_1n_2} \right). $$
    Next, by the Chinese Remainder Theorem, we can write $b_1 = b_1' n_1 + b_1'' n_2$ for $b_1' \mod n_2$ and $b_1'' \mod n_1$ since $\gcd(n_1,n_2)=1$. And likewise for $b_2 \mod n_2$. Hence, similar to the first part of Proposition \ref{G_k(n) conditions}, one gets
    $$ f(n_1n_2, \ell_1 \ell_2) = \sum_{\substack{b_1', b_2' \mod n_2 \\ b_1'' b_2'' \mod n_1}} \left(\frac{b_1'^2+b_2'^2}{n_2} \right) \left(\frac{b_1''^2+b_2''^2}{n_1} \right) e\left( \frac{\ell_1\ell_2 b_1'' + \ell_1\ell_2 b_2''}{n_1} + \frac{\ell_1\ell_2b_1'+ \ell_1\ell_2 b_2'}{n_2} \right) . $$
    We notice that since $\gcd(\ell_1 , n_2)=1$ and $\gcd(\ell_2, n_1)=1$, we can do change of variables $\ell_1 b_1' \to b_1'$, $\ell_1 b_2' \to b_2'$ and $\ell_2 b_1'' \to b_1''$, $\ell_2 b_2'' \to b_2''$. Then using $$\left( \frac{b_1^2+b_2^2}{n} \right) = \left( \frac{(\overline{\ell}b_1')^2+(\overline{\ell}b_2')^2}{n} \right) =\left( \frac{b_1'^2+b_2'^2}{n} \right),  $$ we get that this gives $f(n_1, \ell_1)f(n_2,\ell_2)$.
\end{proof}

\begin{lemma}\label{Properties of rho}
    Let $k_1, k_2\in \Z$. If $m$ and $n$ are relatively prime odd integers, then we have $\rho_{(k_1,k_2)}(mn)= \rho_{(k_1,k_2)}(m)\rho_{(k_1,k_2)}(n)$. Further, for any $\ell$ coprime to $n$, we have $\rho_{(\ell k_1 , \ell k_2)}(n) = \rho_{(k_1,k_2)}(n)$. Moreover, we have the bound $|\rho_{(k_1,k_2)}(n)| \leq \rho(n) \leq \tau(n) n$, where $\rho(n)$ is defined in \eqref{def_of_rho}.
\end{lemma}
\begin{proof}
    The multiplicativity comes from the Chinese Remainder Theorem as in the proof of the previous lemma.
    
    For the second property, write
    $$ \rho_{(\ell k_1 , \ell k_2)}(n) = \sum_{\substack{b_1, b_2 \mod n \\ b_1^2+b_2^2 \equiv 0 \mod n}} e\left( \frac{\ell k_1 b_1+ \ell k_2 b_2}{n} \right).$$
    By a change of variable $\ell b_1 \to b_1$ and $\ell b_2 \to b_2$, we get the same summand without the $\ell$ factors, but with the additional condition that $\ell^2(b_1^2+b_2^2) \equiv 0 \mod n$. Since $\gcd(\ell,n)=1$, this will imply that $\rho_{(\ell k_1 , \ell k_2)}(n) = \rho_{(k_1,k_2)}(n)$.
    
    For the last bounds, note that $|\rho_{(k_1,k_2)}(n)| \leq \rho(n)$ is the trivial bound. For the second bound, by multiplicativity, we just need to study the prime powers. Note that $\rho(p^{\alpha}) = \{(u,v)\in (\Z/p^{\alpha}\Z)^2 : u^2+v^2 \equiv 0 \mod p^{\alpha} \}$. Hence, for each $u\in \Z/p^{\alpha}\Z$, we have at most $2$ choices for $v$, and so $\rho(p^{\alpha}) \leq 2p^{\alpha}$. The result then follows.
\end{proof}

\section{Proof of Theorem \ref{MainTheorem}}

We want to compute
\[\sum_{\substack{d SF\\(d,2N)=1\\w(f\otimes \chi_{8d})=-1}}r(d)L'(1/2,f\otimes\chi_{8d})\psi\left(\frac{8d}{X}\right).\]
This equals
\[\frac{1}{2}\sum_{\substack{d SF\\(d,2N)=1}}r(d)L'(1/2f\otimes\chi_{8d})\psi\left(\frac{8d}{X}\right)-\frac{w(f)}{2}\sum_{\substack{d SF\\(d,2N)=1}}r(d)\chi_{8d}(-N)L'(1/2,f\otimes\chi_{8d})\psi\left(\frac{8d}{X}\right).\]

Opening up $L'(1/2,f\otimes\chi_{8d})$ using Lemma \ref{AFE}, and truncating both sums to $n\leq U=\frac{X}{(\log X)^A}$, for some large $A>0$, we will get
\[L'(1/2,f\otimes\chi_{8d})=\cA_U(f\otimes\chi_{8d})+\cB_U(f\otimes\chi_{8d}),\]
where
\[\cA_U(f\otimes\chi_{8d})=(1-i^{\kappa}\eta \chi_{8d}(N))\sum_{n\geq 1}\frac{\lambda_f(n)\chi_{8d}(n)}{\sqrt{n}}W\left(\frac{n}{U}\right),\]
and
\[\cB_U(f\otimes\chi_{8d})=(1-i^{\kappa}\eta\chi_{8d}(N))\sum_{n\geq 1}\frac{\lambda_f(n)\chi_{8d}(n)}{\sqrt{n}}\left(W\left(\frac{n}{8d}\right)-W\left(\frac{n}{U}\right)\right).\]

\subsection{The $\cA_U$ Sum}
Let $N'=1$ or $N'=N$.
We want to study
\[T(N',h):=\sum_{\substack{d\, SF\\(d,2N)=1}}r(d)\sum_{n\geq 1}\frac{\lambda_f(n)\chi_{8d}(nN')}{\sqrt{n}}h(d,n)\]
where 
\[h(d,n):=W\left(\frac{n}{U}\right)\psi\left(\frac{8d}{X}\right).\]
We detect the SF and coprimality conditions to get
\begin{align*}
T(N',h)&=\sum_{(a_1,2NN')=1}\mu(a_1)\sum_{\substack{a_2|N \\ (a_2,N')=1}} \mu(a_2)\sum_{(d,2)=1}r(da_1^2a_2)\sum_{(n,a_1)= 1}\frac{\lambda_f(n)\chi_{8da_2}(nN')}{\sqrt{n}}h(da_1^2a_2,n)\\
    &=T_1(N',h)+T_{2}(N',h),
\end{align*}
where 
\[T_1(N',h):=\sum_{\substack{(a_1,2NN')=1\\a_1\leq Y}}\mu(a_1)\sum_{\substack{a_2|N \\ (a_2,N')=1}}\mu(a_2)\sum_{(d,2)=1}r(da_1^2a_2)\sum_{(n,a_1)= 1}\frac{\lambda_f(n)\chi_{8da_2}(nN')}{\sqrt{n}}h(da_1^2a_2,n),\]
and
\[T_{2}(N',h):=\sum_{\substack{(a_1,2NN')=1\\a_1> Y}}\mu(a_1)\sum_{\substack{a_2|N \\ (a_2,N')=1}} \mu(a_2)\sum_{(d,2)=1}r(da_1^2a_2)\sum_{(n,a_1)= 1}\frac{\lambda_f(n)\chi_{8da_2}(nN')}{\sqrt{n}}h(da_1^2a_2,n),\]
for $Y=(\log X)^{A/10}$.

The main term will come from the $T_1$ sum, while the $T_2$ sum will be an error.

\subsubsection{Bounding $T_2(N',h)$} 
 We start by adding the coprimality condition $a_2|N$ again, and then writing $d=d_1d_2^2$ for $d_1$ SF and $d_2\in\Z_{>0}$. After exchanging sums and using the fact that $r(d_1d_2^2a_1^2)<\tau(d_1)\tau(d_2^2)\tau(a_1^2)$, we use Cauchy-Schwarz to get
\begin{align*}
    T_2(N',h)\ll \sum_{\substack{(a_1,2NN')=1\\a_1> Y}}\mu(a_1)\tau(a_1^2)\sum_{\substack{(d_2,2N)=1\\d_2\leq \sqrt{X}}}&\tau(d_2^2)\left(\sum_{\substack{(d_1,2N)=1\\d_1 SF\\d_1\leq X/(a_1d_2)^2}} \tau^2(d_1)\right)^{1/2}\\
    &\times \left(\sum_{\substack{(d_1,2N)=1\\d_1 SF\\d_1\leq X/(a_1d_2)^2}}\left\lvert\sum_{(n,d_2)= 1}\frac{\lambda_f(n)\chi_{8d_1}(nN')}{\sqrt{n}}W\left(\frac{n}{U}\right)\right\rvert^{2}\right)^{1/2}.
\end{align*}
Now, employing  \cite[Proposition 3.2]{Kumar} (and in particular the bound in Section 5.5 in their paper) together with the fact that
\[\sum_{n\leq x}\tau^2(n)\asymp x\log^3 x,\]
we get
\[T_{2}(N',h)\ll \frac{X\log^3X}{Y^{1-\epsilon}}.\]

\subsubsection{Analyzing $T_1(N',h)$} We start by changing the order of summations in $T_1(N',h)$. Using $\chi_{8d}(nN')=\left(  \frac{8d}{nN'}\right)$, one has
$$ T_1(N',h)=\sum_{\substack{(a_1,2N)=1\\a_1\leq Y}}\mu(a_1)\sum_{\substack{a_2|N \\ (a_2,N')=1}}\mu(a_2)\sum_{(n,2a_1a_2)= 1}\frac{\lambda_f(n)} {\sqrt{n}}  \sum_{(d,2)=1}r(da_1^2a_2) \left( \frac{8da_1^2a_2}{nN'}\right) h(da_1^2a_2,n) .$$
Next, we use Proposition \ref{Voronoi} to get that the inner sum is
$$  \frac{X}{4a_1^4 a_2^2n^2N'^2} \sum_{(k_1,k_2)\in \Z^2} (-1)^{k_1+k_2} G_{(k_1,k_2)}(a_1^2a_2, nN') \check{h}\left( \frac{(k_1^2+k_2^2)X}{2a_1^4a_2^2n^2N'^2} ,  n \right), $$
where
\[\check{h}\left( \frac{(k_1^2+k_2^2)X}{2a_1^4a_2^2n^2N'^2} ,  n \right)=2\pi \int_0^{\infty}h(Xt,n)J_0\left(2\pi\sqrt{\frac{(k_1^2+k_2^2)Xt}{2a_1^4a_2^2n^2N'^2}}\right)\,dt,\]
is the Bessel transform of $h$.
We can further use Lemma \ref{splitting_G} for splitting the function $G$. Here, the $(k_1,k_2)=(0,0)$ term will give us the main term, and the rest will be an error.\\

\paragraph{\underline{Main Term $(k_1,k_2)=(0,0)$:}}
Denote by $T_{10}(N',h)$ the $(k_1,k_2)=(0,0)$ term. That is,
$$ T_{10}(N',h) = \frac{X}{4N'^2} \sum_{\substack{(a_1,2NN')=1\\a_1\leq Y}}\frac{\mu(a_1)\rho(a_1^2)}{a_1^4}\sum_{\substack{a_2|N \\ (a_2,N')=1 }}\frac{\mu(a_2)\rho(a_2)}{a_2^2}\sum_{(n,2a_1a_2)= 1}\frac{\lambda_f(n) \cdot G(nN')} {n^{5/2}} h_1(n) , $$
where recall
$$ G(nN') = \sum_{b_1,b_2\mod nN'} \left(\frac{b_1^2+b_2^2}{nN'}\right) ,$$
and $$\rho(m)=\#\{ (u_1,u_2)\mod m : u_1^2+u_2^2\equiv 0 \mod m \}$$ are multiplicative functions, and
$$ h_1(n):=\int_0^{\infty} h(Xt,n) dt. $$
Next, we can write this in a cleaner form, since $G(n) \leq n^2$,
\begin{align*}
T_{10}(N',h) = C X \sum_{(n,2)=1}\frac{\lambda_f(n) }{\sqrt{n}}\frac{ G(nN')}{n^2N'^2} \prod_{p|nNN'} \left( 1 - \frac{\rho(p^2)}{p^4}\right)^{-1} \prod_{\substack{p|N \\ p\nmid nN'}} &\left(1  - \frac{\rho(p)}{p^2} \right) \cdot h_1(n) \\
&+ O\left( \frac{X}{Y} \sum_{n=\square}\frac{\tau(n)} {n^{1/2}} |h_1(n)| \right), 
\end{align*}
where
$$C :=\frac{1}{4}  \prod_{p>2} \left( 1- \frac{\rho(p^2)}{p^4} \right). $$
To see how large is the big-O term, we use bounds on $\psi(x)$ and Lemma \ref{W_bound} to get the following bound on the partial derivatives of $h(\cdot,\cdot)$ that follow from integration by parts,
\begin{align}\label{h_bound}
h^{(i,j)}(x,y)\ll_{j,B} \frac{1}{x^iy^j}\left(\log\frac{\kappa N U}{n}\right)\left(1+\frac{n}{\kappa \sqrt{N}U}\right)^{-B}\left(1+\frac{d}{X}
\right)^{-B},\end{align}
  for any $B>0$. Using this in $h_1(n)$, we see that the big-O term is at most $\ll \frac{X}{Y}(\log X)^4$.

Using Proposition \ref{G_k(n) conditions} and noting that $N'\in \{1, N\}$, we get
\begin{align*}
 T_{10}(N',h) = C X \sum_{\substack{(n,2)=1 \\ nN'=\square}}\frac{\lambda_f(n)} {n^{1/2}} \prod_{\substack{p|nNN'}} \left( 1 - \frac{\rho(p^2)}{p^4}\right)^{-1} \prod_{\substack{p|nNN'}} &\left(1  - \frac{\rho(p)}{p^2} \right) h_1(n) \\
 &+ O\left( \frac{X}{(\log X)^{A/100}} \right).
 \end{align*}

We recall that $h(x,y)= W(y/U) \psi(8x/X)$, hence this gives us
\begin{align*}
 T_{10}(N',h)= \frac{C}{16}\hat{\psi}(0) X \sum_{\substack{(n,2)=1 \\ nN'=\square}}\frac{\lambda_f(n)} {n^{1/2}} \prod_{\substack{p|nNN'}}  \left(1  - \frac{\rho(p)}{p^2} \right)&\left( 1 - \frac{\rho(p^2)}{p^4}\right)^{-1}  W \left( \frac{n}{U} \right) \\
 &+ O\left( \frac{X}{(\log X)^{A/100}} \right),
 \end{align*}
where $$\hat{\psi}(0)=16\int_{0}^{\infty}\psi(8t)\,dt$$
is the Fourier transform of $\psi$ at $0$. By the definition of $W(\cdot)$, 
$$ T_{10}(N',h)= \frac{C}{16}  \hat{\psi}(0)X \cdot \frac{1}{2i\pi} \int_{(3)}\frac{\Gamma(u+\kappa/2)}{\Gamma(\kappa/2)}\left(\frac{2\pi}{U\sqrt{N}}\right)^{-u}\frac{\cZ(u)}{u^2} du + O\left( \frac{X}{(\log X)^{A/100}} \right), $$
where
$$ \cZ(u) :=   \sum_{\substack{(n,2)=1 \\ nN'=\square}}\frac{\lambda_f(n)} {n^{1/2+u}} \prod_{\substack{p|nNN'}}  \left(1  - \frac{\rho(p)}{p^2} \right) \left( 1 - \frac{\rho(p^2)}{p^4}\right)^{-1} . $$
 The Dirichlet series $\cZ(u)$ has the following Euler product
\begin{align*}
    \cZ(u) = &\prod_{\substack{p\nmid 2NN' \\ p\equiv 1 \mod 4 }} \left( 1 + \frac{p(p-1)}{p^2-p+1} \left[ \frac{1}{2}\left( 1 - \frac{\lambda_f(p)}{p^{1/2+u}}+ \frac{1}{p^{1+2u}} \right)^{-1} + \frac{1}{2}\left( 1 + \frac{\lambda_f(p)}{p^{1/2+u}}+ \frac{1}{p^{1+2u}} \right)^{-1} - 1\right] \right) \\
    &\prod_{\substack{p\nmid 2NN' \\ p\equiv 3 \mod 4 }} \left( 1 + \frac{p(p+1)}{p^2+p+1}\left[ \frac{1}{2}\left( 1 - \frac{\lambda_f(p)}{p^{1/2+u}}+ \frac{1}{p^{1+2u}} \right)^{-1} + \frac{1}{2}\left( 1 + \frac{\lambda_f(p)}{p^{1/2+u}}+ \frac{1}{p^{1+2u}} \right)^{-1} - 1\right] \right) \\
    &\prod_{\substack{p|NN'}}\frac{p(p^2-\rho(p))}{p^3-\rho(p)} \left[ \frac{1}{2}\left( 1 - \frac{\lambda_f(p)}{p^{1/2+u}} \right)^{-1} + (-1)^{\ord_p(N')}\frac{1}{2}\left( 1 + \frac{\lambda_f(p)}{p^{1/2+u}} \right)^{-1} \right],
\end{align*}
which comes from Hecke recurrences for $\lambda_f(n)$ and the facts that $\rho(p)=(1+\chi_4(p))(p-1)+1$, and $\rho(p^2)=p\cdot \rho(p)$. We can write this in the form
\begin{equation}\label{A(u)}
 \cZ(u)=L(1+2u,\text{sym}^2f) \cdot A(u,N, N'), 
 \end{equation}
where
\begin{align}\label{sym^2}
L(s,\text{sym}^2f)&=\zeta(2s)\sum_{n=1}^{\infty}\frac{\lambda_f(n^2)}{n^s},
\end{align}
is the symmetric square $L$-function for $f$, and
$$A(u,N,N')= \prod_{p\nmid 2NN'} \left( 1+ O\left( \frac{1}{p^{2+2\Re(u)}} + \frac{1}{p^{2+4\Re(u)}}\right) \right) \cdot \prod_{p|2NN'} \left( 1 + O\left( \frac{1}{p^{1+2\Re(u)}} + \frac{1}{p}\right) \right) $$ is an absolutely convergent product for $\Re(u) > -1/4$. Note that $A(u,N, N')$ only depends on $\ord_p(N')$ for $p|NN'$. We have also used that for $\Re(s)>1$,
\[L(s,\text{sym}^2f)=\prod_p\left(1-\frac{\alpha_f(p)^2}{p^s}\right)^{-1}\left(1-\frac{\alpha_f(p)\beta_f(p)}{p^s}\right)^{-1}\left(1-\frac{\beta_f(p)^2}{p^s}\right)^{-1},\]
where $\alpha_f(p)+\beta_f(p)=\lambda_f(p)$ and $\alpha_f(p)\beta_f(p)=1$.

Next, we move the contour in the integral of $T_{10}(N',h)$ to $\Re(u)=-\frac{1}{6}$, and pick up the double pole at $u=0$. This gives us
\begin{align*}
T_{10}(N',h) =  \frac{C}{16} X \cdot \hat{\psi}(0) L(1, \text{sym}^2 f) &A(0,N, N') \left( \log\frac{U\sqrt{N}}{2\pi} + \frac{A'(0,N, N')}{A(0,N, N')} + 2 \frac{L'(1,\text{sym}^2 f)}{L(1,\text{sym}^2 f)} \right) \\
&+ \frac{C}{16} X \cdot \hat{\psi}(0) \frac{1}{2i\pi} \int_{(-1/6)}\Gamma(u+1)\left(\frac{2\pi}{U\sqrt{N}}\right)^{-u} \frac{\cZ(u)}{u^2} du . 
\end{align*}
We use the convexity bound for the $\text{GL}(3)$ object (see \cite[Eqn (5.20)]{Iwk})
$$L(2/3 +it, \text{sym}^2f)  \ll (1+|t|)^{1+\varepsilon},$$ 
to bound the second term. This will give us an error $\ll_{\varepsilon} X^{5/6+\varepsilon}$.

\subsubsection{Off-diagonal terms $(k_1,k_2)\not= (0,0)$}
In this section, we study the following sum
\begin{align*}
    T_{11}(N',h) :=\frac{X}{4}\sum_{(k_1,k_2)\not=(0,0)} (-1)^{k_1+k_2}&\sum_{\substack{(a_1,2N)=1\\a_1\leq Y}}\frac{\mu(a_1) \rho_{(k_1,k_2)}(a_1^2)}{a_1^4}\sum_{\substack{a_2|N}}\frac{\mu(a_2)\rho_{(k_1,k_2)}(a_2)}{a_2^2}\\
    &\times \sum_{(n,2a_1a_2)= 1}\frac{\lambda_f(n)} {\sqrt{n}} \frac{G_{(k_1,k_2)}(nN')}{n^2N'^2} \check{h}\left( \frac{(k_1^2+k_2^2)X}{2a_1^4a_2^2n^2N'^2} ,  n \right),
\end{align*}
where recall
$$ \rho_{(k_1,k_2)}(m)= \sum_{\substack{b_1,b_2 \mod m \\ b_1^2+b_2^2\equiv 0 \mod m}} e \left( \frac{k_1b_1}{m} + \frac{k_2 b_2}{m} \right), $$
is a multiplicative function.
We use Mellin inversion and then a change of variables to write
\begin{align*}
\check{h}(x,y) &= 2\pi \int_{0}^{\infty} h(X t,y) J_0(2\pi\sqrt{tx}) dt \\
&= \frac{1}{iX} \int_0^{\infty} J_0(2\pi \sqrt{xtX^{-1}}) \int_{(-1/4)} \tilde{h}(1+s,y) t^{-s} ds \frac{dt}{t} \\
&=\frac{1}{iX} \int_{0}^{\infty} J_0(t) \int_{(1/2)} \tilde{h}(1-s/2,y) \left( \frac{t}{2\pi \sqrt{xX^{-1}}}\right)^{s} ds \frac{dt}{t} \\
&=\frac{1}{iX} \int_{(1/2)}\tilde{h}(1-s/2,y) \left( \int_{0}^{\infty} J_0(t) t^s \frac{dt}{t} \right) (2\pi \sqrt{xX^{-1}})^{-s} ds \\
&=\frac{X^{-1}}{2i} \int_{(1/2)} \tilde{h}(1-s/2, y) \frac{\Gamma(s/2)}{\Gamma(1-s/2)} (\pi \sqrt{xX^{-1}})^{-s}ds,
\end{align*}
where we used \cite[(17.43.16)]{MR1773820} for the last line. We can further perform one more Mellin inversion in the $y$ variable to get
$$ \check{h}(x,y) = -\frac{X^{-1}}{4\pi} \int_{(1+\epsilon)} \int_{(1/2+\varepsilon)} \tilde{\tilde{h}}(1-s/2,u) y^{-u} (\pi \sqrt{x X^{-1}})^{-s} \frac{\Gamma(s/2)}{\Gamma(1-s/2)} \,du\, ds , $$
where
$$ \tilde{\tilde{h}}(s,u)= \int_{\R^2} h(t,v) t^s v^u \frac{dt}{t} \frac{dv}{v}, $$
and it satisfies using integration by parts several times and the bounds in \eqref{h_bound} on the partial derivatives of $h(\cdot , \cdot )$,
\begin{equation}\label{DoubleMellinbounds}
\tilde{\tilde{h}}(s,u) \ll \frac{X^{\Re(s)} (2U\sqrt{N})^{\Re(u)}}{|u| (1+|s|)^{98} (1+|u|)^{98}} .
\end{equation}

Now, take the $(k_1,k_2)$ sum of  $T_{11}(N',h)$ inside, and split $k_1 = \ell k_1'$ and $k_2=\ell  k_2'$ where $\gcd(\ell,2a_1a_2)=1$, and further define

$$ \cX_{N'}(\alpha, \beta , a_1,a_2,k_1',k_2') := \sum_{(\ell,2a_1a_2)=1} \sum_{(n,2a_1a_2)=1} \frac{\lambda_f(n)}{n^{\alpha} \ell^{\beta}} \cdot \frac{G_{(\ell k_1', \ell k_2')}(nN')}{n^2N'^2}.$$

Since $(\ell , a )=1$ for $a\in \{a_1^2,a_2\}$, we have by Lemma \ref{Properties of rho} that $\rho_{(\ell k_1',\ell k_2')}(a)= \rho_{(k_1',k_2')}(a)$, and using a change of variable, we get 
\begin{align*}
&T_{11}(N',h) = -\frac{1}{16\pi} \sum_{\substack{(a_1,2NN')=1\\a_1\leq Y}}\frac{\mu(a_1)}{a_1^4}\sum_{\substack{a_2|N \\ (a_2,N')=1}}\frac{\mu(a_2)}{a_2^2}  \sum_{\substack{(k_1',k_2')\not=(0,0) \\ p | \gcd(k_1',k_2') \implies p|2a_1a_2}} (-1)^{k_1'+k_2'} \rho_{(k_1',k_2')}(a_1^2) \rho_{(k_1',k_2')}(a_2) \\
&\times \int_{(1/2+\epsilon)} \int_{(1+\varepsilon)} \tilde{\tilde{h}}(1-s/2,u+s) \left( \frac{a_1^2 a_2 N'}{\sqrt{k_1'^2+k_2'^2}} \right)^s \frac{ \Gamma(s/2)}{\pi^{s}\cdot \Gamma(1-s/2)} \cX_{N'}(1/2 + u, s, a_1,a_2,k_1',k_2')\, ds\, du.
\end{align*}

We want to shift the contours in the latter integrals, and for this we need to understand the analytic behaviour of $\cX$. We do this in the following lemma.

First, for an $L$-function $L(s)$ and a positive integer $a$, denote by $L_a(s)$ the function formed from the same Euler product as $L(s)$ but with the primes dividing $a$ removed. Note that for $\Re(s)>0$,
\begin{equation}\label{L_func_ineq}
| L_a(s) | \leq \tau(a) |L(s)|.
\end{equation}

\begin{lemma}\label{EvaluatingZ}
Let $k_1,k_2\in \Z$, and $a_1,a_2$ be positive integers such that $p|\gcd(k_1,k_2) \implies  p|2a_1a_2$. We have that
\begin{align*}
 \cX_{N'}(\alpha,\beta , a_1,a_2,k_1,k_2) = \zeta_{2a_1a_2\gcd(k_1,k_2)NN'}(\beta ) \cdot &L_{2a_1a_2\gcd(k_1,k_2)NN'}\left(1+\alpha,f \otimes \chi_{(k_1,k_2)} \right)\\
 &\times Z_{ N'}(\alpha, \beta, a_1,a_2,k_1,k_2) ,
 \end{align*}
where $\chi_{(k_1,k_2)}$ is a primitive Dirichlet character of conductor $q|k_1^2+k_2^2$, and $Z_{N'}$ is given by some Euler product absolutely convergent in $\Re(\alpha)\geq -1$ and $\Re(\beta)\geq 3/2 + \varepsilon$, and uniformly bounded in terms of $N$ and $N'$.
\end{lemma}

\begin{proof}
   From Lemma \ref{jointlymultiplicative}, we know that $G_{(\ell k_1,\ell k_2)}(n)$ is jointly multiplicative in $n$ and $\ell$, hence we can study the Euler product
    $$ \left( \prod_{\substack{p| N' \\ p\nmid 2a_1a_2}} \sum_{n, \ell \geq 0} \frac{\lambda_f({p^n})}{p^{\alpha n+\beta \ell}} \frac{G_{(p^{\ell} k_1, p^{\ell}k_2)}(p^{n+\nu_p(N')})}{p^{2n+2\nu_p(N')}} \right) \cdot \left( \prod_{\substack{p \nmid N' \\ p\nmid 2a_1a_2}} \sum_{n, \ell \geq 0} \frac{\lambda_f({p^n})}{p^{\alpha n+\beta \ell}} \frac{G_{(p^{\ell} k_1, p^{\ell}k_2)}(p^n)}{p^{2n}} \right) . $$
    For now, we focus on $p\nmid N'$. For the generic prime $p\nmid 2a_1a_2\gcd(k_1,k_2)$, we have for $\ell=0$ by Proposition \ref{G_k(n) conditions} that the terms that survive are
    $$ 1+\frac{\lambda_f(p)\left(  \frac{k_1^2+k_2^2}{p}\right)}{p^{\alpha+1}}.$$
    Since $p$ is odd, we can also write $\left(\frac{k_1^2+k_2^2}{p}\right) = \chi_{(k_1, k_2)}(p) 1_{(k_1^2+k_2^2,p)=1} ,$ where $ \chi_{(k_1,k_2)}(p)$ is some primitive character of conductor $q|k_1^2+k_2^2$. For $\ell=1$, the only terms that appear are $n=0$ and $n=2$, which give a contribution of 
    \[\frac{1}{p^{\beta}}+\frac{\lambda_f(p^2)\left(\frac{-1}{p}\right)}{p^{2\alpha+\beta+2}}.\]
    While for $\ell= 2$, one gets
\[\frac{1}{p^{2\beta}}+\frac{\lambda_f(p^2)(p^2-\rho(p))}{p^{2\alpha+2\beta+2}}+\frac{\lambda_f(p^3)\left(\frac{k_1^2+k_2^2}{p}\right)}{p^{3\alpha+2\beta+1}}.\]
We start detecting a pattern for even and odd $\ell$'s, and so we group all the terms in a way to get that the full contribution from all $\ell$'s and $n$'s in the case $p\nmid 2a_1a_2\gcd(k_1,k_2)NN'$ is
\[1+\sum_{\ell=1}^{\infty}\frac{1}{p^{\ell\beta}}+\sum_{r=1}^{\infty}\lambda_f(p^{2r})\Bigg[\frac{\left(\frac{-1}{p}\right)}{p^{2r\alpha+(2r-1)\beta+2}}+\frac{p^2-\rho(p)}{p^{2r\alpha+2}}\sum_{\ell=2r}^{\infty}\frac{1}{p^{\ell\beta}}\Bigg]+\left(\frac{k_1^2+k_2^2}{p}\right)\sum_{r=0}^{\infty}\frac{\lambda_f(p^{2r+1})}{p^{(2r+1)\alpha+2r\beta+1}}.\]
   For $\Re(\alpha) \geq -1$ and $\Re(\beta)\geq 3/2+\varepsilon$, this is
    \begin{align*}
    &1+\sum_{\ell\geq 1}\frac{1}{p^{ \beta\ell}} +\frac{\lambda_f(p)\left(  \frac{k_1^2+k_2^2}{p}\right)}{p^{\alpha+1}}+ O\left( \frac{1}{p^{2\Re(\alpha)+\Re(\beta)+2}} + \frac{1}{p^{2\Re(\alpha) + 2\Re(\beta)}}\right) \\
    &=1+ \sum_{\ell\geq 1}\frac{1}{p^{\beta \ell}} +\frac{\lambda_f(p)\left(  \frac{k_1^2+k_2^2}{p}\right)}{p^{\alpha+1}}+ O( p^{-(1+2\varepsilon)}).
    \end{align*} 
   Now we can extract the  $\zeta_{2a_1a_2\gcd(k_1,k_2)NN'}(\beta )$ and $L_{2a_1a_2\gcd(k_1,k_2)NN'}(1+\alpha,f\otimes\chi_{(k_1,k_2)})$ factors appearing in the statement of the lemma, from the first three terms of the previous line.
    
    In the case $p|\gcd(k_1,k_2)$, we don't get a contribution since $p\nmid 2a_1a_2$ by definition of $\cX$, and $p|\gcd(k_1,k_2) \implies p|2a_1a_2$.

    In the case $p| NN'$, we get an Euler product of the form
    $$ \prod_{p| NN'} \left( 1 + O\left( p^{O(1)} \right) \right) \ll_{N, N'} 1 , $$
    which we put in the $Z_{N'}$ factor.

\end{proof}

In view of Lemma \ref{EvaluatingZ}, we first split the sum over $(k_1',k_2')$ into two parts $k_1'^2+k_2'^2 \leq W$ and $k_1'^2+k_2'^2 > W$ where $W=U^2 X^{-1} Y^4$. For the former range, we shift the contour in $s$ to $\Re(s)=\frac{7}{4}$. For the latter range, we shift the contour in $s$ to $\Re(s)=\frac{5}{2}$. In both cases, we shift the contour in $u$ to $\Re(u)=-1$.

For the range $k_1'^2+k_2'^2 \leq W$, we get
\begin{gather*}
 -\frac{N'}{16\pi} \sum_{\substack{(a_1,2NN')=1\\a_1\leq Y}}\frac{\mu(a_1)}{a_1^4}\sum_{\substack{a_2|N \\ (a_2,N')=1}}\frac{\mu(a_2)}{a_2^2}  \sum_{\substack{(k_1',k_2')\not=(0,0) \\ p | \gcd(k_1' ,k_2') \implies p|2a_1a_2 \\ k_1^2+k_2^2 \leq W}}(-1)^{k_1'+k_2'}\rho_{(k_1',k_2')}(a_1^2) \rho_{(k_1',k_2')}(a_2) \\
 \times \int_{(-1)} \int_{(7/4)} \tilde{\tilde{h}}(1-s/2,u+s) \left( \frac{a_1^2 a_2 N'}{\sqrt{k_1'^2+k_2'^2}} \right)^s \frac{ \Gamma(s/2)}{\pi^{s}\cdot \Gamma(1-s/2)} \cX_{N'}(1/2 + u, s, a_1,a_2,k_1',k_2')\, ds\, du. 
\end{gather*}
We proceed to bound this. 
Using \eqref{DoubleMellinbounds} and \eqref{L_func_ineq}, together with the bound $\rho(n)\leq n\tau(n)$ (Lemma \ref{Properties of rho}) and the fact that $\tau(n)$ is submultiplicative, we get the bound
\begin{gather*}
 \ll_{N'} X^{1/8} U^{3/4} \sum_{\substack{(a_1,2NN')=1\\a_1\leq Y}}\tau(a_1)^{O(1)} a_1^{3/2} \sum_{\substack{a_2|N \\ (a_2,N')=1}} \tau(a_2)^{O(1)}a_2^{3/4} \int_{(-1)} \int_{(3/2)} \frac{|\Gamma(s/2)|}{|\Gamma(1-s/2)|} \\
 \times \sum_{\substack{(k_1',k_2')\not=(0,0) \\ p | \gcd(k_1' ,k_2') \implies p|2a_1a_2 \\ k_1'^2+k_2'^2 \leq W}} \frac{\tau(k_1'^2+k_2'^2)^{O(1)}L(\frac{1}{2}, f \otimes \chi_{(k_1',k_2')}) }{(k_1'^2+k_2'^2)^{7/8}} \frac{\, ds\, du}{|u|(1+|s|)^{98}(1+|u|)^{98}} .
\end{gather*}
We want to bound the sum inside the integral first. To do this, note that for $k_1'^2+k_2'^2=m^2d$ where $d$ is square-free, we have $\chi_{(k_1',k_2')}(n) = \chi_d(n)$ for $(n,m)=1$. Hence we get, by using \eqref{L_func_ineq} again, that
\begin{gather*}
\sum_{\substack{(k_1',k_2')\not=(0,0) \\ p | \gcd(k_1' ,k_2') \implies p|2a_1a_2 \\ k_1'^2+k_2'^2 \leq W}} \frac{\tau(k_1'^2+k_2'^2)^{O(1)}  L(\frac{1}{2} + it, f\otimes \chi_{(k_1',k_2')}) }{(k_1'^2+k_2'^2)^{7/8}} \\
\ll \sum_{m \leq \sqrt{W}} \frac{\tau(m)^{O(1)}}{m^{7/4}} \sum_{d\leq W/m^2} r(m^2d) \tau(m^2 d)^{O(1)} \frac{|L(\frac{1}{2} + it, f \otimes \chi_d)|}{d^{7/8}}.
\end{gather*}
Now, using the submultiplicativity of $\tau(n)$ and the property that $r(m^2d)\leq r(m^2)r(d) \leq 4 \tau(m)^2 r(d)$, we obtain that the latter is bounded by
$$ \ll \sum_{m \leq \sqrt{W}} \frac{\tau(m)^{O(1)}}{m^{7/4}} \sum_{d\leq W/m^2} r(d) \tau(d)^{O(1)} \frac{|L(\frac{1}{2} + it, f \otimes \chi_d)|}{d^{7/8}}. $$
We can now do a dyadic decomposition for the inner sum and then perform Cauchy-Schwarz. After this, we are left to use the work of Li \cite{XLi}, i.e. Lemma \ref{XLi_bound}, to bound the second moment of the twisted $L$-function of $f$.
Putting all of the above together, the sum over $(k_1',k_2')$ finally gives
$$ \ll_{N} W^{1/8} (\log X)^{O(1)}.$$

Hence, at the end, we get the bound
$$ \ll_{N} X^{1/8} U^{3/4} W^{1/8} Y^{5/2} (\log X)^{O(1)}. $$
By a similar calculation in the range $k_1'^2+k_2'^2 > W$, we get
$$\ll_{N} \frac{U^{3/2} Y^4 (\log X)^{O(1)}}{X^{1/4}W^{1/4}} .$$
Using the definitions of $U=\frac{X}{(\log X)^A}, Y=(\log X)^{A/10}$, and $W=U^2X^{-1}Y^4$, the off-diagonal terms give an error of size $$\ll_{N } U \cdot  Y^3 (\log X)^{O(1)} \ll_{N,A} \frac{X}{(\log X)^{A/100}}. $$

\subsection{The $\cB_U$ sum}
We want to estimate

\[\sum_{\substack{d\, SF\\(d,2N)=1}}r(d)\cB_U(f\otimes\chi_{8d})\psi\left(\frac{8d}{X}\right). \]
We bound this using Cauchy-Schwarz, and then employing \cite[Proposition 4.3]{zhou}\footnote{Note that Proposition 4.3 of \cite{zhou} also covers the case $f=g$ in the notation there. This is similar to the work of Li \cite{XLi} where he obtained the same result but for the $L$-functions instead of their derivatives, see \cite[Proposition 6.1]{XLi}.}, together with
\[\sum_{d\leq x}r(d)^2\asymp x \log x,\]
to get that the sum is
\[\ll X\sqrt{\log X}(\log\log X)^3.\]

\nocite{*}

\bibliographystyle{amsplain}

\bibliography{Ref}

\end{document}